\documentclass[11pt]{amsart}
\usepackage[english]{babel}
\usepackage[latin1]{inputenc}
\usepackage[dvips,final]{graphics}
\usepackage{amsmath,amsfonts,amssymb,amsthm,amscd,array,stmaryrd,mathrsfs}
\usepackage{url}
 \usepackage[all]{xy}
\usepackage{graphicx}
\usepackage{color}              % Need the color package
\usepackage{epsfig}
\usepackage{psfrag}
\usepackage{epstopdf}

\usepackage{textcomp}

\let\ssection=\section
\renewcommand{\section}{\setcounter{equation}{0}\ssection}

\theoremstyle{plain}
\newtheorem{thm}{Theorem}%[section]
\newtheorem*{thm*}{Theorem}
\newtheorem{lem}{Lemma}[section]

\newtheorem{conj}{Conjecture}

\theoremstyle{definition}

\newtheorem{rem}[lem]{Remark}
\newtheorem{ex}[lem]{Example}

\newcommand{\R}{\mathbb{R}}
\newcommand{\Z}{\mathbb{Z}}

\newcommand{\bbF}{\mathbb{F}}

\newcommand{\bbH}{\mathbb{H}}

\newcommand{\bbO}{\mathbb{O}}

\newcommand{\N}{\mathbb{N}}

\newcommand{\A}{\mathcal{A}}
\newcommand{\I}{\mathcal{I}}

\newcommand{\B}{\mathcal{B}}
\newcommand{\F}{\mathcal{F}}

\newcommand{\GL}{\mathrm{GL}}

\newcommand{\gS}{\mathfrak{S}}

\def\a{\alpha}
\def\b{\beta}
\def\d{\delta}

\def\om{\omega}
\def\r{\rho}

\begin{document}

\title[]{Extremal set theory, cubic forms on $\bbF_2^n$ and\\
 Hurwitz square identities}

\author{Sophie Morier-Genoud
\and
Valentin Ovsienko}\thanks{ovsienko@math.univ-lyon1.fr, tel:+33666824954}

\address{
Sophie Morier-Genoud,
Sorbonne Universit\'es, UPMC Univ Paris 06, UMR 7586, Institut de Math\'ematiques de Jussieu- Paris Rive Gauche, Case 247, 4 place Jussieu, F-75005, Paris, France}

\address{
Valentin Ovsienko,
CNRS,
Laboratoire de Math\'ematiques, 
Universit\'e de Reims-Champagne-Ardenne, 
FR 3399 CNRS, F-51687, Reims, France
}

\email{sophie.morier-genoud@imj-prg.fr,
ovsienko@math.univ-lyon1.fr}

\date{}

\begin{abstract}
We consider a family, $\F$, of subsets of an $n$-set
such that the cardinality of the symmetric difference of any two elements $F,F'\in\F$
is not a multiple of $4$.
We prove that the maximal size of $\F$ is bounded by $2n$, unless $n\equiv{}3\mod4$
when it is bounded by $2n+2$.
Our method uses cubic forms on $\bbF_2^n$
and the Hurwitz-Radon theory of square identities.
We also apply this theory to obtain some information about
Boolean cubic forms and so-called additive quadruples.
\end{abstract}

\maketitle

\thispagestyle{empty}

%%%%%%%%%%%%%%%%%%%%%%%%%%%%%%%%%
\section{Introduction and the main results}
%%%%%%%%%%%%%%%%%%%%%%%%%%%%%%%%%

In this note, we link different subjects:
extremal set theory, Boolean cubic forms,
non-associative algebras and
the Hurwitz theory of ``square identities''.

Let $\F$ be a family of subsets of $\{1,2,\ldots,n\}$.
For $F,F'\in\F$, define the {\it symmetric difference}
$$
F\oplus{}F':=\left(F\setminus{}F'\right)\cup{}\left(F'\setminus{}F\right).
$$
Denote by $d(F,F')$ the cardinality of $F\oplus{}F'$,
which is sometimes called the {\it Hamming distance} between
the sets $F$ and $F'$.
The following is our most elementary statement.

\begin{thm}
\label{ScratchF}
If for every distinct $F,F'\in\F$, the distance
$d(F,F')$ is not a multiple of~$4$,
then
$$
|\F|\leq
\left\{
\begin{array}{lll}
2n,&
n\equiv0,1,2&\mod4\\[4pt]
2n+2,&
n\equiv3&\mod4.
\end{array}
\right.
$$
This bound is sharp.
\end{thm}
\noindent
Note that replacing $4$ by another integer, say $3$ or $5$, 
the bound for the size of $\F$ becomes quadratic in $n$.

Theorem~\ref{ScratchF} belongs to the vast domain
of extremal set theory,
see~\cite{FW} and~\cite{VV} for an overview.
The classical Oddtown Theorem  states: {\em if the cardinality
of every $F\in\F$ is odd, and that of every intersection $F\cap{}F'$ is even, then 
$|\F|\leq{}n$}, see~\cite{BF,VV} and references therein.
It is well-known that the bound remains $n$ if one switches ``odd'' and ``even'',
but if one replaces ``odd'' by ``even'' or ``even'' by ``odd'', then
the bound becomes exponential in $n$.
In Theorem~\ref{ScratchF}, we impose
{\it a priori} no restriction on the members of the family~$\F$.

Theorem~\ref{ScratchF} is related to the Oddtown Theorem by the formula
\begin{equation}
\label{DisF}
d(F,F')
=|F|+
|F'|-
2\,|F\cap{}F'|.
\end{equation}
This suggests an idea to replace intersection of sets
by symmetric difference, and parity condition by double parity condition.
The Oddtown Theorem directly implies the upper bound $|\F|\leq2n+2$
for every $n$ and $|\F|\leq2n$ for $n\equiv0\;(\!\!\!\!\mod4)$.
However, it seems that Theorem~\ref{ScratchF} 
cannot be entirely deduced from the Oddtown Theorem by elementary methods.

We will use linear algebra over the field $\bbF_2=\{0,1\}$,
replacing $\F$ by a subset $A\subset\bbF_2^n$.
Symmetric difference of sets then corresponds to
sum of vectors.
However, unlike the case of the Oddtown Theorem, the proof 
of Theorem~\ref{ScratchF} is not reduced to linear algebra.
Using cubic forms on $\bbF_2^n$, we deduce Theorem~\ref{ScratchF}
from the celebrated Hurwitz-Radon theorem~\cite{Hur,Rad}.

The second main goal of this note is to study invariants of
cubic forms on $\bbF_2^n$.
The following statement is a 
strengthening of Theorem~\ref{ScratchF}.

\begin{thm}
\label{ScratchFCorBis}
(i)
Given a cubic form $\a:\bbF_2^n\to\bbF_2$ and a subset $A\subset\bbF_2^n$,
assume that for every $x\not=x'\in{}A$ one has $\a(x+x')=1$, then $|A|\leq\r(2^n)$,
where $\r$ is the classical Hurwitz-Radon function.

(ii)
This bound is sharp, at least in the cases $n\equiv1,2$ or $3\;(\!\!\!\!\mod4)$.
\end{thm}

Classification of cubic forms on~$\bbF_2^n$ is a fascinating problem
which is solved only for $n\leq9$; see~\cite{Hou1,BL,Lan}.
The maximal cardinality of a subset $A\subset\bbF_2^n$ such that
$\a\,\big|_{(A+A)\setminus\{0\}}\equiv1$ is an interesting characteristic of a cubic form $\a$.
It resembles the Arf invariant of quadratic forms, but of course is not 
enough for classification.

The paper is organized as follows.
First, we interpret Hurwitz sum of square identities 
in terms of extremal set theory in$t\bbF_2^n$.
This interpretation uses the Euclidean norm in some
non-associative algebras.
We then provide a construction of extremal sets reaching
the upper bound of Theorem~\ref{ScratchF}.

%%%%%%%%%%%%%%%%%%%%%%%%%%%%%%%%%
\section{Cubic forms and square identities}\label{CuT}
%%%%%%%%%%%%%%%%%%%%%%%%%%%%%%%%%

%%%%%%%%%%%%%%%%%%%%%%%%%%%%%%%%%
\subsection{Hurwitz identities}
%%%%%%%%%%%%%%%%%%%%%%%%%%%%%%%%%
A sum of square identity of size $[r,s,N]$ is an identity of the form
$$
(a_1^2+\cdots{}+a_r^2)\,(b_1^2+\cdots{}+b_s^2)
=c_1^2+\cdots{}+c_N^2,
$$
where $c_i$ are bilinear expressions in  $a_j$ and $b_k$ with
coefficients in $\Z$.
In~\cite{Hur1}, Hurwitz formulated his famous problem 
to determine all the triples $(r,s,N)$ such that there exists an identity
of size $[r,s,N]$.
The problem remains widely open, see~\cite{Sha} for a survey.

The {\it Hurwitz-Radon function} $\r$ is a function on the set of natural numbers
$\r:\N\to\N$.
If $N=2^n(2m+1)$, then $\r(N)=\r(2^n)$
(i.e., it depends only on the dyadic part of $N$), and the latter number is 
given by
$$
\r(2^n)=\left\{ \begin{array}{lcll}
2n+1, \quad&n\equiv& 0 &\mod 4\\
2n, \quad&n\equiv& 1,2 &\mod 4\\
2n+2, & n\equiv &3 &\mod 4.
\end{array}
\right.
$$

The celebrated Hurwitz-Radon theorem~\cite{Hur,Rad};
see also~\cite{Sha}, is formulated as follows:
{\it there exists an identity of size $[r,N,N]$ if and only if $r\leq\r(N)$}.
This is the only case where the Hurwitz problem is solved.

%%%%%%%%%%%%%%%%%%%%%%%%%%%%%%%%%
\subsection{Cubic forms}
%%%%%%%%%%%%%%%%%%%%%%%%%%%%%%%%%
A {\it cubic form} on $\bbF_2^n$ is a function
$\a:\bbF_2^n\to\bbF_2$ of the form
$$
\a(x)=\sum_{1\leq{}i\leq{}j\leq{}k\leq{}n}\a_{ijk}\,
x_ix_jx_k,
$$
where $x=(x_1,\ldots,x_n)$ and where $\a_{ijk}\in\{0,1\}$.
Note that, over $\bbF_2$, we have $x_i^2=x_i$ and therefore
every cubic polynomial can be viewed as a {\it homogeneous} cubic form.

Consider the following cubic function:
\begin{equation}
\label{NashAlp}
\a_\bbO(x)=\sum_{1\leq{}i<j<k\leq{}n}
x_ix_jx_k+
\sum_{1\leq{}i<j\leq{}n}\,x_ix_j+
\sum_{1\leq{}i\leq{}n}x_i.
\end{equation}
The function $\a_\bbO$ is a {\it counting function}.
This means, it is invariant with respect to the action of the group
of permutations on the coordinates and depends only on the {\it Hamming weight}\footnote{
Recall that the Hamming weight of $x$ is the number of components $x_i=1$.}
of~$x$ that we denote by ${\mathrm wt}(x)$.
More precisely,
$$
\a_\bbO(x)=\left\{
\begin{array}{lll}
0,& \hbox{if }\; {\mathrm wt}(x)\equiv0& \mod4\\[4pt]
1,&\hbox{otherwise}.
\end{array}
\right.
$$

%%%%%%%%%%%%%%%%%%%%%%%%%%%%%%%%%
\subsection{Twisted group algebras}\label{CuT}
%%%%%%%%%%%%%%%%%%%%%%%%%%%%%%%%%

Let $f:\bbF_2^n\times\bbF_2^n\to\bbF_2$ be a function
of two variables.
The {\it twisted group algebra} associated to $f$ is the real
$2^n$-dimensional algebra denoted by $\left(\R[\bbF_2^n],f\right)$,
with basis $\{e_x\,|\,x\in\bbF_2^n\}$ and the product given by
$$
e_x\cdot{}e_{x'}=(-1)^{f(x,x')}\,e_{x+x'}.
$$
This algebra is, in general, neither commutative nor associative.
The non-commutativity is measured by the function
$$
\b(x,y):=f(x,y)+f(y,x),
$$
while the non-associativity is measured by the function
$$
\d{}f(x,y,z):=f(y,z)+f(x+y,z)+f(x,y+z)+f(x,y).
$$
Many classical algebras, such as the algebras of quaternions
$\bbH$, of octonions $\bbO$, and, more generally, the Clifford algebras
and the Cayley-Dickson algebras, can be realized as twisted group
algebras over $\bbF_2^n$; see~\cite{AM}.

%%%%%%%%%%%%%%%%%%%%%%%%%%%%%%%%%
\subsection{From cubic forms to algebra}\label{FCS}
%%%%%%%%%%%%%%%%%%%%%%%%%%%%%%%%%
There exists an interesting subclass of twisted group algebras
characterized by a cubic function in one variable,
instead of the function $f$ in two variables.
It was introduced and studied in~\cite{MGO},
and we give a very short account here.

Given a cubic form $\a$,
there exists a (unique modulo coboundary) ``twisting function'' $f$ satisfying the conditions:
\begin{itemize}
\item[(a)] 
First polarization formula:
$$
\b(x,y)= \a(x+y)+\a(x)+\a(y).
$$
\item[(b)] 
Second polarization formula:
$$
\d{}f(x,y,z)=\a(x+y+z)+\a(x+y)+\a(x+z)+\a(y+z)+\a(x)+\a(y)+\a(z).
$$
\item[(c)] Linearity of $f$ in 2nd variable:
$$f(x,y+y')=f(x,y)+f(x,y').$$
\item[(d)] 
Reconstruction of $\a$ from $f$:
$$
f(x,x)=\a(x).
$$
\end{itemize}

The existence of $f$ follows from an explicit formula. 
We replace every monomial in $\a$
according to the following rule:
\begin{equation}
\label{Fakset}
\begin{array}{rcl}
x_ix_jx_k&\longmapsto&x_ix_jy_k+x_iy_jx_k+y_ix_jx_k,\\[4pt]
x_ix_j&\longmapsto&x_iy_j,\\[4pt]
x_i&\longmapsto&x_iy_i.
\end{array}
\end{equation}
where $i<{}j<{}k$,
and obtain this way a function $f$ in two arguments,
satisfying properties~(a)-(d).

In particular, the cubic form $\a_\bbO$ generates the following
the twisting function:
$$
f_\bbO(x,y)=\sum_{1\leq{}i<j<k\leq{}n}
\left(
x_ix_jy_k+x_iy_jx_k+y_ix_jx_k
\right)+
\sum_{1\leq{}i\leq{}j\leq{}n}\,x_iy_j.
$$
The obtained twisted group algebras, denoted by $\bbO_n$,
generalize the classical algebra $\bbO$ of octonions 
(which is isomorphic to $\bbO_3$).

\begin{rem}
One cannot choose a polynomial of degree $\geq4$,
instead of a cubic function, in order to construct
a twisting function $f$ satisfying properties (a)-(d).
Indeed, let us apply the differential $\d$ to the equation in property (b).
Since $\d^2=0$, one obtains after a short computation:
$$
\begin{array}{rl}
0=&
\a(x+y+z+t)\\[4pt]
&+\a(x+y+z)+\a(x+y+t)+\a(x+z+t)+\a(y+z+t)\\[4pt]
&+\a(x+y)\!+\!\a(x+z)\!+\!\a(x+t)\!+\!\a(y+z)\!+\!\a(y+t)\!+\!\a(z+t)\\[4pt]
&+\a(x)+\a(y)+\a(z)+\a(t).
\end{array}
$$
This is exactly the condition that $\a$ is a polynomial of degree at most $3$.
 \end{rem}

%%%%%%%%%%%%%%%%%%%%%%%%%%%%%%%%%
\subsection{From algebra to square identities}
%%%%%%%%%%%%%%%%%%%%%%%%%%%%%%%%%

Consider a twisted group algebra $\left(\R[\bbF_2^n],f\right)$,
its elements are of the form
$$
a=\sum_{x\in\bbF_2^n}a_x\,e_x,
$$
with coefficients $a_x\in\R$.
Define the Euclidean norm by
$$
\|a\|^2:=\sum_{x\in\bbF_2^n}a_x^2.
$$
Consider two sets $A,B\subset\bbF_2^n$ and
the coordinate subspaces $\A$ and $\B\subset\left(\R[\bbF_2^n],f\right)$:
$$
\big\{a\;|\;a=\sum_{x\in{}A}a_x\,e_x
\big\}
\quad\hbox{and}\quad
\big\{b\;|\;b=\sum_{y\in{}B}b_y\,e_y
\big\}.
$$

The condition
\begin{equation}
\label{NorC}
\|a\|^2\,\|b\|^2=\|a\,b\|^2,
\end{equation}
gives a square identity of size $[|A|,|B|,|A+B|]$.

Consider the twisted group algebra $\left(\R[\bbF_2^n],f\right)$ corresponding to a cubic function $\a$
as explained in Section~\ref{FCS}.
It turns out that the condition~\eqref{NorC} can be very easily expressed in terms of the form $\a$.

The following statement is proved in~\cite{MGO,LMO,Des}.
\begin{lem}
\label{NashLem}
The condition~\eqref{NorC} is equivalent to the following:
for all $x\not=x'\in{}A$ and $y\not=y'\in{}B$ such that $x+x'=y+y'$,
one has
$
\a(x+x')=1.
$
\end{lem}

We will apply the above lemma in the case $\a=\a_\bbO$,
and choosing $B=\bbF_2^n$.
The condition on the set $A$ then reads 
${\mathrm wt}(x+x')$ is not a multiple of $4$,
for all distinct $x,x'\in{}A$.

%%%%%%%%%%%%%%%%%%%%%%%%%%%%%%%%%
\section{Construction of Hurwitzian sets}\label{HurSec}
%%%%%%%%%%%%%%%%%%%%%%%%%%%%%%%%%

In this section, we construct examples of sets $A\subset\bbF_2^n$
of cardinality $|A|=\r(2^n)$ satisfying the condition:
${\mathrm wt}(x+x')$ is not a multiple of $4$,
for all distinct $x,x'\in{}A$.
Such sets were already considered in~\cite{LMO}
where they were called {\it Hurwitzian sets}.
In particular, we discuss a relation to
the binary Hadamard matrices.

%%%%%%%%%%%%%%%%%%%%%%%%%%%%%%%%%
\subsection{Cases $n\equiv1,2\,(\!\!\!\mod4)$}
%%%%%%%%%%%%%%%%%%%%%%%%%%%%%%%%%
In this case, $\r(2^n)=2n$.
The following choice of a Hurwitzian set is perhaps the most 
obvious.
Choose the following set:
$$
A=\left\{
0,\;e_1,\;e_2,\,\ldots,\,e_n,\;e_1+e_2,\;e_1+e_3,\,\ldots,\,e_1+e_n
\right\}.
$$
For all $x,x'\in{}A$, the weight of the sum satisfies ${\mathrm wt}(x+x')\leq3$, and thus
$\a_\bbO(x+x')=1$, provided $x+x'\not=0$.
Therefore $A$ is a Hurwitzian set.

Note that the above choice is not unique.
However, it is easy to see that the set $A$ is 
the only Hurwitzian set which is a ``shift-minimal downset''
according to the terminology of~\cite{GT1}.

%%%%%%%%%%%%%%%%%%%%%%%%%%%%%%%%%
\subsection{Case $n\equiv3\,(\!\!\!\mod4)$}
%%%%%%%%%%%%%%%%%%%%%%%%%%%%%%%%%
In this case, $\r(2^n)=2n+2$ which is the most interesting
situation for many reasons.

Consider the element of maximal weight:
\begin{equation}
\label{Omeq}
\om=(1\,1\,\ldots\,1)=e_1+\cdots+e_n.
\end{equation}
One can choose the above set $A$, completed by $\om$ and $e_1+\om$.
Let us give a more symmetric example.

Choose the set
$$
A=\left\{
0,\;\om,\;e_1,\;e_2,\,\ldots,\,e_n,\;e_1+\om,\;e_2+\om,\,\ldots,\,e_n+\om
\right\}.
$$
The weight of a non-zero element of the sumset $A+A$ can be one of the
following four values: $1,2,n-1$, or $n-2$.
Since this is never a multiple of $4$, we conclude that $A$ is
a Hurwitzian set.
Moreover, it is not difficult to show that the above set is
the only Hurwitzian set invariant with respect to the group
of permutations $\gS_n$. 

%%%%%%%%%%%%%%%%%%%%%%%%%%%%%%%%%
\subsection{Another choice in the case $n\equiv3\,(\!\!\!\mod8)$, relation to the Hadamard matrices}\label{HadSec}
%%%%%%%%%%%%%%%%%%%%%%%%%%%%%%%%%

The case $n\equiv3\,(\!\!\!\mod8)$ is a subcase of the above one.
Remarkably, there is a choice of Hurwitzian set
based on the classical Hadamard matrices.

Recall that a {\it Hadamard matrix} is an 
$m\times{}m$-matrix $H$ with entries $\pm1$,
such that ${}^tHH=m\I$, where ${}^tH$ is the transpose
of $H$ and $\I$ is the identity matrix.
It is known that a Hadamard matrix can exist only if $m=1,2$ or $m=4s$;
existence for arbitrary $s$ is the classical Hadamard conjecture.

The construction is as follows.
We remove the first column of $H$ and construct two $(m-1)\times{}m$-matrices,
$H_1,H_2$ with entries $0,1$.
The matrix $H_1$ is obtained by replacing $1$ by $0$ and $-1$ by $1$,
the matrix $H_2$ is obtained by replacing $-1$ by $0$.

\begin{lem}
The rows of $H_1$ and $H_2$ form a Hurwitzian set in $\bbF_2^{4s-1}$,
provided $s$ is odd.
\end{lem}

\begin{proof}
It follows from the definition of a Hadamard matrix that
every sum of two distinct rows of $H_1$ is of weight $2s$,
and similarly for $H_2$.
The sum of a row of $H_1$ with a row of $H_2$ is of weight 
$2s-1$ or $4s-1$.
\end{proof}

\begin{ex}
\label{HamEx}
The (unique up to equivalence) $12\times12$ Hadamard matrix $H$
corresponds to the following $12\times11$ binary matrices:
$$
H_1=
\tiny{
\left(
\begin{array}{rrrrrrrrrrr}
1&1&1&1&1&1&1&1&1&1&1\\[4pt]
0&1&0&1&1&1&0&0&0&1&0\\[4pt]
0&0&1&0&1&1&1&0&0&0&1\\[4pt]
1&0&0&1&0&1&1&1&0&0&0\\[4pt]
0&1&0&0&1&0&1&1&1&0&0\\[4pt]
0&0&1&0&0&1&0&1&1&1&0\\[4pt]
0&0&0&1&0&0&1&0&1&1&1\\[4pt]
1&0&0&0&1&0&0&1&0&1&1\\[4pt]
1&1&0&0&0&1&0&0&1&0&1\\[4pt]
1&1&1&0&0&0&1&0&0&1&0\\[4pt]
0&1&1&1&0&0&0&1&0&0&1\\[4pt]
1&0&1&1&1&0&0&0&1&0&0
\end{array}
\right)}
\qquad
H_2=
\tiny{
\left(
\begin{array}{rrrrrrrrrrr}
0&0&0&0&0&0&0&0&0&0&0\\[4pt]
1&0&1&0&0&0&1&1&1&0&1\\[4pt]
1&1&0&1&0&0&0&1&1&1&0\\[4pt]
0&1&1&0&1&0&0&0&1&1&1\\[4pt]
1&0&1&1&0&1&0&0&0&1&1\\[4pt]
1&1&0&1&1&0&1&0&0&0&1\\[4pt]
1&1&1&0&1&1&0&1&0&0&0\\[4pt]
0&1&1&1&0&1&1&0&1&0&0\\[4pt]
0&0&1&1&1&0&1&1&0&1&0\\[4pt]
0&0&0&1&1&1&0&1&1&0&1\\[4pt]
1&0&0&0&1&1&1&0&1&1&0\\[4pt]
0&1&0&0&0&1&1&1&0&1&1
\end{array}
\right)}
$$
(which are related to the extended Golay code).
The rows of the matrices $H_1$ and $H_2$ constitute a Hurwitzian set
in $\bbF_2^{11}$ of cardinality $24$.
\end{ex}

An idea of further development is to understand
the relations of Theorem~\ref{ScratchF} to doubly even binary codes.
Existence of such relations is indicated by the above example
where the celebrated Golay code appears
explicitly.

\subsection{Case $n\equiv0\,(\!\!\!\mod4)$}
Recall that $\r(2^n)=2n+1$ in this case.
However, we will show in the next section that there are
no Hurwitzian sets in the case.
Moreover, we are convinced that a similar situation holds for
any cubic form.

\begin{conj}
Given a Boolean cubic function $\a$ on $\bbF_2^n$ with $n\equiv0\,(\!\!\!\mod4)$,
there is no set $A$ such that $\a\,\big|_{(A+A)\setminus\{0\}}\equiv1$ and
$|A|=2n+1$.
\end{conj}

\noindent
This conjecture is easily verified for $n=4$,
as well as for $\a=\a_\bbO$ and arbitrary $n$.

%%%%%%%%%%%%%%%%%%%%%%%%%%%%%%%%%
\section{Proof of Theorems \ref{ScratchF} and \ref{ScratchFCorBis}}\label{PruS}
%%%%%%%%%%%%%%%%%%%%%%%%%%%%%%%%%

Let us prove Theorem \ref{ScratchFCorBis}.
Fix an arbitrary cubic form $\a$, and
let $A\subset{}\bbF_2^n$ be a set such that
$$
\a\;\big|\;_{(A+A)\setminus\{0\}}\equiv1.
$$
Lemma~\ref{NashLem} then implies
$\|a\|\,\|b\|=\|ab\|$ for all $a\in\A$ and arbitrary $b$.
We therefore obtain a square identity of size $\left[|A|,\,2^n,\,2^n\right]$.
The Hurwitz-Radon Theorem implies that $|A|\leq\r(2^n)$.
This bound is sharp as follows from the constructions of
Hurwitzian sets; see Section~\ref{HurSec}.
Theorem~\ref{ScratchFCorBis} follows.

Fixing $\a=\a_\bbO$, we obtain the statement of
Theorem \ref{ScratchF} in the cases where $n\equiv1,2,3\;(\!\!\!\mod4)$.
In the last case $n=4m$, Theorem~\ref{ScratchFCorBis} implies that $|A|\leq2n+1$.
It remains to show that if~$n=4m$, then
$|A|\leq2n$.

Suppose that $n=4m$ and $A$ is a Hurwitzian set.
Every element $x\in{}A$ can be replaced by $\widetilde{x}=x+\om$,
where $\om$ is the ``longest'' element~(\ref{Omeq}).
Indeed, one has
$$
x+x'=x+x'+\om.
$$
Replacing $x$ by $\widetilde{x}$ whenever $x_n=1$,
we obtain another Hurwitzian set, $\widetilde{A}$, such that
$x_n=0$ for all $x\in\widetilde{A}$.
But then $|\widetilde{A}|\leq\rho(2^{n-1})=2n$.

Theorem \ref{ScratchF} is proved.

\begin{rem}
Note that the Oddtown Theorem easily implies the upper bound $|\F|\leq2n+2$
for every $n$.
Indeed, let $\F_0\subset\F$ be the family of subsets of even cardinality,
without loss of generality we assume that~$\F_0$ contains the empty set.
It follows from the assumption of Theorem~\ref{ScratchF},
that for every $F(\not=\emptyset)\in\F_0$,
one has $|F|\equiv2\;(\!\!\!\!\mod4)$, 
and for every two elements,
$|F\cap{}F'|\equiv1\;(\!\!\!\!\mod2)$.
The Oddtown Theorem then implies $|\F_0|\leq{}n+1$.
Similarly, $|\F_1|\leq{}n+1$, where $\F_1$ is the odd subfamily of $\F$.
\end{rem}

%%%%%%%%%%%%%%%%%%%%%%%%%%%%%%%%%
\section{Additive quadruples}
%%%%%%%%%%%%%%%%%%%%%%%%%%%%%%%%%

Our next statement concerns so-called {\it additive quadruples}.
If $A,B\subset\bbF_2^n$, four elements $x,x'\in{}A,\,y,y'\in{}B$
form an additive quadruple $(x,x',y,y')$ if
$$
x+x'+y+y'=0.
$$
We call an additive quadruple {\it proper} if $x\not=x'$ and $y\not=y'$.

\begin{thm}
\label{JCor}
Let $A,B\subset\bbF_2^n$ with
$|A|\leq|B|$.
If every proper additive quadruple $(x,x',y,y')$ satisfies
$\a(x+x')=1$, then  
$
|A+B|\geq{}\Omega(|A|^{\frac{6}{5}}).
$
\end{thm}

\begin{proof}
Fix, as above, an arbitrary cubic form $\a$.
Suppose that $A$ and $B$ are two subsets of same cardinality $|A|=|B|=r$, 
and such that for all proper additive quadruples
$(x,x',y,y')$ one has $\a(x+x')=\a(y+y')=1$.
One obtains an identity
of size $[r,r,N]$, where $N=|A+B|$.
 The Hurwitz problem is still open in this particular case and even
an asymptotic of the least value $N_{\rm min}$ as a function of $r$ is not known exactly.
However, it is known that asymptotically
$$
C_1\,r^{\frac{6}{5}}\leq{}N_{\rm min}(r)\leq{}C_2\,\frac{r^2}{\log(r)}.
$$
where $C_1$ and $C_2$ are some constants.
The upper bound follows easily from the Hurwitz-Radon theorem,
and the lower bound was recently obtained in~\cite{HWY},
which is precisely the statement of Theorem~\ref{JCor}.
\end{proof}

The Balog-Szemer\'edi-Gowers theorem~\cite{BS,Gow},
in the $\bbF_2^n$ case (see~\cite{GT}) states, roughly speaking,
that the sumset $A+B$ grows slowly, provided there are ``many''
additive quadruples (of order~$|A|^3$).
The above result is a sort of converse statement.

\medskip

\noindent \textbf{Acknowledgments}.
This project 
 was partially supported by the PICS05974 ``PENTAFRIZ'' of CNRS.
We are grateful to John Conway for stimulating discussions.

%%%%%%%%%%%%%%%%%

\end{document}